\documentclass{amsart}
\usepackage{amssymb,amscd,amsthm,amsxtra}

\newtheorem{theorem}{Theorem}[section]
\newtheorem{lemma}[theorem]{Lemma}

\newtheorem{prop}[theorem]{Proposition}
\theoremstyle{definition}
\newtheorem{definition}[theorem]{Definition}
\theoremstyle{remark}

\numberwithin{equation}{section}

\newcommand{\R}{{\mathbb R}}

\newcommand{\eps}{\varepsilon}
\renewcommand{\epsilon}{\varepsilon}
\newcommand{\Per}{{\rm {Per_\sigma}}}
\newcommand{\Om}{\Omega}
\newcommand{\p}{\partial}

\renewcommand{\le}{\leqslant}
\renewcommand{\leq}{\leqslant}
\renewcommand{\ge}{\geqslant}
\renewcommand{\geq}{\geqslant}

\def\Xint#1{\mathchoice
{\XXint\displaystyle\textstyle{#1}}%
{\XXint\textstyle\scriptstyle{#1}}%
{\XXint\scriptstyle\scriptscriptstyle{#1}}%
{\XXint\scriptscriptstyle\scriptscriptstyle{#1}}%
\!\int}
\def\XXint#1#2#3{{\setbox0=\hbox{$#1{#2#3}{\int}$ }
\vcenter{\hbox{$#2#3$ }}\kern-.6\wd0}}
\def\avg{\Xint{\mbox{{\bf ---}}}}

\begin{document}

\title[A fractional perimeter-Dirichlet integral functional]{Minimization of \\
a fractional perimeter-Dirichlet integral functional}

\author[L. Caffarelli]{Luis Caffarelli}
\address{University of Texas at Austin,
Department of Mathematics\\
1 University Station,
C1200 Austin, TX 78712-1082 (USA)}
\email{caffarel@math.utexas.edu}

\author[O. Savin]{Ovidiu Savin}
\address{Columbia University,
Mathematics Department\\
2990 Broadway,
New York, NY 10027 (USA)}
\email{savin@math.columbia.edu}

\author[E. Valdinoci]{Enrico Valdinoci}
\address{Weierstra{\ss} Institut f\"ur Angewandte Analysis und Stochastik\\
Mohrenstra{\ss}e 39, D-10117 Berlin (Germany)}
\address{Universit\`a di Milano, Dipartimento di Matematica\\
Via Cesare Saldini 50, I-20133 Milan (Italy)}
\email{enrico.valdinoci@unimi.it}

\begin{abstract}
We consider a minimization problem
that combines the Dirichlet energy with
the nonlocal perimeter of a level set, namely
$$ \int_\Om |\nabla u(x)|^2\,dx+\Per\Big( \{u > 0\},\Om \Big),$$
with $\sigma\in(0,1)$. We obtain regularity results for
the minimizers and for their free boundaries $\p \{ u>0\}$ using blow-up analysis. We will also give related results about
density estimates, monotonicity formulas, Euler-Lagrange equations
and extension problems.
\end{abstract}

\maketitle

\section{Introduction}

Let $\Om$ be a bounded domain~$\R^n$ and ~$\sigma\in(0,1)$ a fixed parameter. In this paper we discuss regularity properties for minimizers of the energy functional
\begin{equation}\label{E v}
J(u):=\int_\Om |\nabla u|^2 dx + \Per (E, \Om), \quad \quad \quad  \mbox{$E=\{u>0\}$ in $\Omega$.}
\end{equation}
where $\Per (E, \Om)$ represents the $\sigma$-fractional perimeter of the set $E$ in $\Omega$.

Here the set $E$ is fixed outside $\Om$ and coincides with $\{u>0\}$ in $\Om$, and we minimize $J$ among all functions~$u\in H^1(\Om)$ with prescribed boundary data i.e. $u=\varphi$ on~$\partial \Om$ for some fixed $\varphi \in H^1(\Om)$.

The fractional perimeter functional $\Per(E,\Om)$ was first introduced in~\cite{crs} and it represents the $\Om$-contribution in the double integral of the norm $\|\chi_E\|_{H^{\sigma/2}}$. Precisely, for any measurable set~$E\subseteq \R^n$
\begin{equation}\label{P}
\Per(E,\Om):=L(E\cap \Om,E^c)+L(E\setminus \Om, \Om\setminus E),
\end{equation}
where
$$ L(A,B):=\int_{A\times B}\frac{dx\,dy}{|x-y|^{n+\sigma}}.$$
It is known (see~\cite{cv, adm, maz, dpv}) that up to multiplicative constants ~$\Per(E,\R^n)$ converges to the classical perimeter functional as~$\sigma\to 1$
and it converges to $|E|$, the Lebesgue measure of $E$, as $\sigma \to 0$.
In this spirit, the functional in~\eqref{E v}
formally interpolates between the two-phase free boundary problem treated in \cite{ac} (where the term $\Per(E,\Omega)$ is replaced
by the classical perimeter of $E$ in $\Omega$)
and the Dirichlet-perimeter minimization functional treated in~\cite{acks}
(where $\Per(E,\Omega)$ is replaced by the Lebesgue measure of $E$ in $\Omega$).

In fact, all previous models correspond to particular cases of the general
nonlocal  phase transition setting as discussed in~\cite{chen} (see
in particular Section~3.5 there):
in our case, the square of the 
$H^{\sigma/2}$
norm of the function
${\,\rm sign\,} u$ is, in terms of~\cite{chen},
the double convolution of the
``phase field parameter'' $\phi$
with the corresponding fractional Laplacian
kernel.

The existence of minimizers follows easily by the direct method in the calculus
of variations, see Lemma \ref{l1} below. Our first regularity result deals with the H\"older regularity of solutions and density estimates for the free boundary $\p E$.

\begin{theorem}\label{H}
Let~$(u,E)$ be a minimizer of $J$ in~$B_1$ with $0 \in \p E$.
Then~$u$ is~$C^{\alpha}(B_1)$, with~$\alpha:=1-\frac\sigma2)$ and
\begin{equation}\label{3bis}
\|u\|_{C^\alpha(B_{r_0})}\le C.\end{equation}
Moreover for any $r \le r_0$
\begin{equation}\label{dense} \min\Big\{ |B_r\cap E|,\;
|B_r\cap E^c|\Big\} \ge c r^n.\end{equation}
The positive constants $C$, $c$ above depend only on $n$ and $\sigma$, and $r_0$ depends also on $\|u\|_{L^2(B_1)}$.
\end{theorem}

We remark that the H\"older exponent obtained in Theorem~\ref{H}
is consistent with the natural scaling of the problem, namely
\begin{equation}\label{scaling}\begin{split}
&{\mbox{if $u$ is a minimizer and~$u_r(x):=r^{\frac\sigma2-1}u(rx)$,}}
\\&\quad {\mbox{
then~$u_r$ is also a minimizer.}}
\end{split}\end{equation}

A minimizer $u$ is harmonic in its positive and negative sets and formally, at points $x$ on the free boundary $\{u=0\}$ it satisfies
\begin{equation}\label{EL0}
\kappa_{\sigma}(x):=\int_{\R^n} \frac{\chi_{E^c} -\chi_E}{|x-y|^{n+\sigma}} \, \,  dy =|\nabla u^+(x)|^2-|\nabla u^-(x)|^2,
\end{equation}
where $\kappa_\sigma(x)$ represents the $\sigma$-fractional curvature of $\p E$ at $x$ (a precise statement will be given in Theorem \ref{flat}).

Generically, we expect that the minimizer $u$ is Lipschitz near the free boundary. 
Then the fractional curvature becomes the dominating term in the free boundary condition above and $\p E$ can be viewed as a perturbation of the $\sigma$-minimal surfaces which were treated in \cite{crs}. 
However, differently from the limiting cases $\sigma=0$ and $\sigma=1$, for $\sigma \in (0,1)$ it seems difficult to obtain the Lipschitz continuity of $u$ at all points (see the discussion at the end of Section 5). 
For the regularity of the free boundary we use instead a monotonicity formula and study homogenous global minimizers. 
Following the strategy in \cite{crs} we obtain an improvement of flatness theorem for the free boundary $\p E$. 
We also show in the spirit of ~\cite{CVcone, CVgen} that in dimension $n=2$ all global minimizers are trivial and by the standard dimension reduction argument we obtain the following result.

\begin{theorem}\label{FB reg}
Let~$(u,E)$ be a minimizer in~$B_1$.
Then $\p E$ is a $C^{1,\gamma}$-hypersurface and it satisfies the Euler-Lagrange equation \eqref{EL0} in the viscosity sense, outside a small singular set $\Sigma \subset \p E$ of Haussdorff $(n-3)$-dimension.
\end{theorem}

In particular in dimension $n=2$ the free boundary is always a $C^{1,\gamma}$ curve. 
We remark that by using the strategy in \cite{barrios} the $C^{1,\gamma}$ regularity of $\p E$ can be improved to $C^\infty$ regularity.

The proofs of Theorem~\ref{H} and~\ref{FB reg}
require some additional results, that will be presented
in the course of the paper, such as
a monotonicity formula, a precise formulation of the Euler-Lagrange equation
and an equivalent extension problem of local type.

The paper is organized as follows. In Section 2 we state various estimates for the change in the Dirichlet integral whenever we perturb the set $E$ by $E \cup A$. We use these estimates throughout the paper and their proofs are postponed in the last section of the paper. We prove Theorem \ref{H} in Section 3 and the improvement of flatness theorem in Section 4. The monotonicity formula and some of its consequences are presented in Section 5. Finally in Section 6 we prove Theorem \ref{FB reg} by showing the regularity of cones in dimension 2.

\section{Estimates for the harmonic replacement}

In order to rigorously deal with the minimization concept
of the functional in~\eqref{E v}, we introduce some notation.

Let $\varphi \in H^1(\Omega)$ and $E_0 \subset \Omega ^c$ be given. We want to minimize the energy
\begin{equation}\label{J}
J_\Om(u):=\int_\Om |\nabla u|^2 dx + \Per (E, \Om)
\end{equation}
among all {\it admissible pairs} $(u,E)$ that satisfy
$$u-\varphi \in H_0^1(\Omega), \quad \quad E \cap \Omega^c=E_0,$$
$$u \ge 0 \quad \mbox{a.e. in $E \cap \Om$}, \quad \quad u \le 0 \quad \mbox{a.e. in $E^c \cap \Om$.} $$
We assume that there is an admissible pair with finite energy, say for simplicity $J(\phi,E_0 \cup \{\phi \ge 0\}) < \infty.$ 
From the lower semicontinuity of $J$ we easily obtain the existence of minimizers.

\begin{lemma}\label{l1}
There exists a minimizing pair $(u,E)$.
\end{lemma}

\begin{proof}
Let $(u_k,E_k)$ be a sequence of pairs along which $J$ approaches its infimum. By compactness, after passing to a subsequence, we may assume that
$u_k \rightharpoonup u  $ in $H^1(\Om)$, $u_k \to u$ in $L^2(\Om)$ and $\chi_{E_k} \to \chi_E$ in $L^1(\Om)$. Then $(u,E)$ is admissible and by the lower semicontinuity of the fractional perimeter functional (i.e. Fatou's lemma) we obtain that $(u,E)$ is a minimizing pair.
\end{proof}

Notice that a minimizing pair in $\Om$ is also a minimizing pair in any subdomain of $\Om$. We assume throughout, after possibly modifying $E$ on a set of measure $0$, that the topological boundary of $E$ coincides with its essential boundary, that is
$$\p E=\left \{x\in \R^n {\mbox{ s.t. }} 0< |E \cap B_r(x)| < |B_r(x)|
\quad \mbox{for all $r>0$}  \right \}.$$

We recall the notion of {\it harmonic replacement} from \cite{acks}.

\begin{definition}\label{H.R} Let $\varphi \in H^1(\Om)$
and $K \subset \Om$ be a measurable set. Assume that the set
$$\mathcal D :=\{v {\mbox{ s.t. }} v-\varphi \in H_0^1(\Om) \quad \mbox{and $v=0$ a.e. in $K$} \} $$
  is not empty. Then we denote by $\varphi_K \in \mathcal D$ the unique minimizer of
 $$\min_{v \in \mathcal D} \int_{\Om}|\nabla v|^2,$$
 and say that $\varphi_K$ is the harmonic replacement of $\varphi$ that vanishes in $K$.
\end{definition}

{F}rom the definition it follows that
$$\int_{\Om} \nabla \varphi_K \cdot \nabla w =0, \quad \quad \quad \mbox{for all $w \in H_0^1(\Om)$ with $w=0$ a.e. in $K$.}  $$
Also, it is straightforward to check that if $\varphi \ge 0$ then $\varphi_K$ is subharmonic. In this case we think that $\varphi_K$ is defined pointwise as the limit of its solid averages.

Clearly if $(u,E)$ is a minimizing pair then we obtain
$$u^+=u^+_{E^c} \quad \mbox{and} \quad u^-=u^-_{E}.$$

Below we estimate the difference in the Dirichlet energies of the harmonic replacements in two different sets $E$ and $E \setminus A$, in terms of the measure of the set $A\subset B_{3/4}$. These estimates depend on the geometry of $E$ and $A$. We assume that $\varphi \in H^1(B_1) \cap L^\infty(B_1)$, $\varphi \ge 0$, and let
$$w:=\varphi_{E^c}, \quad \quad v:=\varphi_{E^c \cup A}.$$
The first lemma deals with the case when $A$ is interior to a ball.
\begin{lemma}\label{Dbe2.bis}
Assume $v$, $w$ are as above and $A:=B_\rho \cap E$ for some $\rho \in [\frac 14, \frac 34]$. Then
$$\int_{B_1} |\nabla v|^2-|\nabla w|^2 \, \, dx \le C |A| \, \|w\|^2_{L^\infty(B_1)},$$
for some constant $C$ depending only on $n$.
\end{lemma}

The next lemma gives the same bound in the case when $A$ is exterior to a ball under the additional hypothesis that $A$ satisfies a density property.

\begin{lemma}\label{Dbe2}
Let $v$, $w$ be as above and assume $E \cap B_{1/2}=\emptyset$. Let $A \subset B_{3/4} \setminus B_{1/2}$ be a closed set that satisfies the density property
\begin{equation*}
|A \cap B_r(x)| \ge \beta r^n \quad \mbox{for all $x \in \p A$ and $B_r(x) \cap B_{1/2}= \emptyset$},\end{equation*}
for some $\beta>0$. Then
$$\int_{B_1} |\nabla v|^2-|\nabla w|^2 \, \, dx\le C(\beta) |A| \, \|w\|^2_{L^\infty(B_1)},$$
for some constant $C(\beta)$ depending only on $n$ and $\beta$.
\end{lemma}

Finally we provide a more precise estimate in the case when $\p E$ is more regular.

Let $u\in H^1(B_1) \cap C(\overline \Om) $ be harmonic in the sets $E=\{u>0\}$ and $\{u<0\}$. Assume $$0 \in \p E \quad \mbox{and} \quad  E=\{x_n > g(x')\}$$ is given by the subgraph in the $e_n$ direction of a $C^{1,\gamma}$ function. For a sequence of $\eps_k \to 0$ we consider sets $$A_k:= \{g(x') < x_n < f_k(x')\} \subset B_{\eps_k},$$
for a sequence of functions $f_k$ with bounded $C^{1,\gamma}$ norm. For each $k$ we define $\bar u_k$ the perturbation of $u$ for which the positive set is given by $E\cup A_k$, i.e.
$$\bar u^+_k= u^+_{E^c \setminus A_\eps} \quad \bar u^-_k=u^-_{E\cup A}.$$

\begin{lemma}\label{c1g} Then
$$\lim_{k \to \infty} \frac{1}{|A_k|} \int_{B_1} |\nabla \bar u_k|^2 -|\nabla u|^2 \, \, dx= |\nabla u^-(0)|^2-|\nabla u^+(0)|^2.$$
\end{lemma}

The proofs of Lemmas \ref{Dbe2.bis}-\ref{c1g} will be completed in the last section.

\section{Proof of Theorem~\ref{H}}\label{78:Y}

In this section we obtain the H\"older continuity of minimizers and uniform density estimates for their free boundary.
We adapt to our goals the strategy of \cite{acks}, and we simplify some steps using Lemma ~\ref{Dbe2.bis}. We start with a density estimate.

\begin{lemma}\label{d88d99.99}
Let~$(u,E)$ be a minimizer in~$B_1$ and assume $$0\in\partial E \quad \mbox{and} \quad \|u^+\|_{L^\infty(B_1)} \le M,$$ for some constant $M$. Then
$$ |E \cap B_{1/2}| \ge \delta, \quad \quad \|u^-\|_{L^\infty(B_{1/2})} \le K, $$
for some positive constant $\delta$, $K$ depending on $n$, $\sigma$ and $M$.
\end{lemma}

\begin{proof} First we prove the density estimate. For each $ \rho \in[\frac 14 , \frac 34]$, set
$$V_ \rho=|E\cap B_\rho|, \quad a(\rho)=\mathcal H ^{n-1}(E \cap \p B_\rho).$$
and assume by contradiction that $V_{1/2} < \delta$ small.

For each such $\rho$ we consider $\bar u$ the perturbation of $u$ which has as positive set $E \setminus A$ with $A:=E \cap B_\rho$, that is
$$\bar u^+:=u^+_{E^c \cup A}, \quad \bar u^-:=u^-_{E \setminus A} .$$
{F}rom the minimality of $(u,E)$ we find
\begin{equation}\label{ei}
\Per(E,B_1) -\Per (E \setminus A, B_1) \le \int_{B_1} |\nabla \bar u|^2 -|\nabla u|^2  dx.
\end{equation}
Since (see \eqref{eleq1})
\begin{align}\label{44}
\int_{B_1} |\nabla \bar u|^2 -|\nabla u|^2  dx&
=\int_{B_1} |\nabla \bar u^+|^2 -|\nabla u^+|^2  dx - \int_{B_1} |\nabla (\bar u^--u^-)|^2 \, dx \\
\nonumber & \le \int_{B_1} |\nabla \bar u^+|^2 -|\nabla u^+|^2  dx,
\end{align}
we use Lemma ~\ref{Dbe2.bis} and the definition of $\Per$ (see \eqref{P}) and we conclude that
$$L(A, E^c)-L(A, E \setminus A) \le C M^2 |A|. $$
Hence
\begin{equation} \label{41}
L(A,A^c) \le 2 L(A, E \setminus A) + C M^2 |A| \le 2 L(A,B_\rho^c) + C M^2 V_{\rho} .
\end{equation}
We estimate the left term by applying Sobolev inequality 
(see, e.g., Theorem~7 in~\cite{sv}): we obtain that
$$V_\rho^\frac{n-\sigma}{n}=\|\chi_A\|_{L^\frac{2n}{n-\sigma} (\R^n)}^2  \le C \| \chi_A \|^2_{H^{\sigma/2}(\R^n)}=C L(A, A^c).$$
If $x \in B_\rho$ then
$$\int_{\mathcal B_\rho^c} \frac{1}{|x-y|^{n+\sigma}}dy \le C \int_{\rho-|x|}^\infty\frac{1}{r^{n+\sigma}}r^{n-1}d r \le
C(\rho-|x|)^{-\sigma},$$
hence integrating in the set $A$ we obtain
$$L(A,B_\rho^c)\le C \int_0^\rho a(r)(\rho-r)^{-\sigma} \, dr.$$
We use these inequalities into \eqref{41} and the assumption that $V_\rho \le \delta$ is sufficiently small to find
$$V_\rho^{\frac{n-\sigma}{n}} \le C \int_0^\rho a(r)(\rho-r)^{-\sigma} \, dr.$$
Integrating the inequality above between $\frac 14$ and $t\in [\frac 14 , \frac 12]$ gives
\begin{equation}{\label{e2.6}}
\int_{1/4}^t V_\rho ^{\frac{n-\sigma}{n}}d\rho \le C t^{1-\sigma} \int_0^t a(r)dr \le C V_t.
\end{equation}
The proof is now a standard De Giorgi iteration: let
$$t_k=\frac{1}{4}+\frac{1}{2^k},\ \ \ \ v_k=V_{t_k},$$ and notice that
$t_2=\frac12$ and $t_\infty=\frac{1}{4}$. Equation (\ref{e2.6}) yields
$$2^{-(k+1)}v_{k+1}^\frac{n-\sigma}{n} \le C v_k.$$
Since $v_2 < \delta$, that is conveniently small, we obtain $v_k \to 0$ as $k \to \infty$. Thus $V_{1/4}=0$ and we contradict that $0 \in \p E$.

For the bound on $u^-$ we write the energy inequality for $\rho=\frac 34$ and we estimate also the negative term in \eqref{44} by Poincare inequality
$$\int_{B_1}|\nabla (\bar u^--u)|^2 dx \ge c \int_{B_1}|\bar u^- -u^-|^2 dx \ge c \int_{E \cap B_{1/2}}|\bar u^-|^2 dx \ge c \delta (\sup_{B_{1/2}}\bar u^-)^2,$$
where in the last inequality we used that $u^-$ is harmonic in $B_{3/4}$.

We have
$$0 \le L(A,E^c) \le L(A,E\setminus A) + CM^2 V_\rho - c \delta (\sup_{B_{1/2}} \bar u^-)^2 $$
and the desired conclusion follows since
\begin{equation*}
L(A, E\setminus A) \le L(B_\rho, B_\rho^c) \le C,  \quad V_\rho \le C, \quad \mbox{and} \quad u^- \le \bar u^-.
\qedhere\end{equation*}
\end{proof}

If $(u,E)$ is a minimizing pair in $B_r$ then the rescaled pair $(u_r,E_r)$ is minimizing in $B_1$ with
\begin{equation}\label{res}
u_r(x):=r^{\frac \sigma 2 -1} u(rx), \quad \quad E_r:=r^{-1} E.
\end{equation}
Let $$\lambda_r^+:=\|u_r^+\|_{L^\infty(B_1)}=r^{\frac \sigma 2 -1}\|u^+\|_{L^\infty(B_r)},$$
and define $\lambda_r^-$ similarly.

If either $\lambda_r^+$ or $\lambda_r^-$ is less than $1$ then, by Lemma \ref{d88d99.99} with $M=1$,
 $$\lambda_{r/2}^+\le C, \quad \lambda_{r/2}^- \le C, \quad \mbox{and} \quad c \le \frac{|E\cap B_r|}{|B_r|} \le 1-c, $$ with $c$, $C$ constants depending on $\sigma$ and $n$. Theorem \ref{H} follows provided the inequalities above hold for all small $r$. Thus, in order to prove Theorem \ref{H} it remains to show that for all $r \le r_0$ either $\lambda_r^+ \le 1$ or $\lambda_r^- \le 1$. This follows from the next lemma which is a consequence of the
Alt-Caffarelli-Friedman monotonicity formula in \cite{acf}.

\begin{lemma}\label{oa}
Let $(u,E)$ be a minimizing pair in $B_1$, and assume $0 \in \p E$. Then
$$\lambda^+_r  \, \lambda_r^- \le C r^\sigma \|u\|_{L^2(B_1)}^2, \quad \quad \quad \forall r \in (0, 1/4],$$
with $C$ depending only on $n$.
\end{lemma}

\begin{proof}
Similar arguments appear in Section 2 of \cite{acks}. We sketch the proof below.

First we prove that $u^+$ and $u^-$ are continuous. For this we need to show that $u^+=u^-=0$ on $\p E$.
Assume by contradiction that, say for simplicity $u^-(0)>0$. Since
$$\limsup_{x \to 0}u^-(x)=u^-(0),$$
we see that the density of $E$ in $B_r$ tends to $0$ as $r\to 0$. Since $u^+\ge 0$ is subharmonic and $u^+=0$ a.e. in $E^c$ it follows that $u^+$ must vanish of infinite order at the origin. Then $\lambda_r^+ \le 1$ for all small $r$ and by the discussion above $E$ has positive density in $B_r$ for all small $r$ and we reach a contradiction.

Since $u^+$ and $u^-$ are continuous subharmonic functions with disjoint supports we can apply  Alt-Caffarelli-Friedman monotonicity formula,
according to which
$$\Psi(r):=\frac{1}{r^4}\int_{B_r}\frac{|\nabla u^+|^2}{|x|^{n-2}}\,dx
\int_{B_r}\frac{|\nabla u^-|^2}{|x|^{n-2}}\,dx,$$
is increasing in $r$.

{F}rom the definition of the harmonic replacement it follows that (see Lemma 2.3 in \cite{acks} for example) $$\triangle (u^+)^2=2|\nabla u^+|^2$$ and we find
$$c \| u^+ \|_{L^\infty(B_{r/2})}^2 \le c \, \, \avg_{B_r} (u^+)^2 dx \le \int_{B_r}\frac{|\nabla u^+|^2}{|x|^{n-2}}\,dx \le C \, \avg_{B_{2r}} (u^+)^2 dx.$$
We use these bounds in the monotonicity formula above and obtain the conclusion.~\end{proof}

\section{Improvement of flatness for the free boundary}

In this section we obtain the Euler-Lagrange equation at points on the free boundary and also we show that if $\p E$ is sufficiently flat in some ball $B_r$ then $\p E$ is a $C^{1,\gamma}$ graph in $B_{r/2}$. The proofs are similar to the corresponding proofs for nonlocal minimal surfaces in \cite{crs}. The difference is that when we perturb $E$ by a set $A$, the change in the nonlocal perimeter is bounded by the change in the Dirichlet integrals (instead of $0$), and by Section 2, this can be bounded in terms of $|A|$.

Our main theorem on this topic is the following.

\begin{theorem}{\label{flat}}
Assume $(u,E)$ is minimal in $B_1$ and that in $B_1$
$$\{x_n > \eps_0\} \subset E \subset \{x_n > -\eps_0 \}, \quad \quad \|u\|_{L^\infty} \le 1,$$
for some $\eps_0>0$ small depending
on $\sigma$ and $n$. Then $\p E \cap B_{1/2}$ is a $C^{1, \gamma}$ graph in the $e_n$
direction and it satisfies the Euler-Lagrange equation in the viscosity sense
\begin{equation}\label{EL}
\int_{\R^n} \frac{\chi_{E^c}-\chi_E}{|y|^{n+\sigma}} \, dy = |\nabla u^+(x)|^2-|\nabla u^-(x)|^2, \quad \quad x \in \p E.
\end{equation}
\end{theorem}

The constant $\gamma$ above depends on $n$ and $\sigma$. The Euler-Lagrange equation in the viscosity sense means that at any point $x$ where $\p E$ has a tangent $C^2$ surface included in $E$ (respectively $E^c$) we have $\ge$ (respectively $\le$) in \eqref{EL}.

First we bound the $\sigma$-curvature of $\p E$ at points $x$ that have a tangent ball from $E^c$.

\begin{lemma}\label{ELvs}
Let $(u,E)$ be a minimizing pair in $B_1$. Assume that $B_{1/4}(-e_n/4)$ is tangent from exterior to $E$ at $0$. Then
$$\int_{\R^n} \frac{\chi_{E^c}-\chi_E}{|x|^{n+\sigma}} \, dx \le C \, \|u^+\|^2_{L^\infty(B_1)}$$
with $C$ depending on $n$ and $\sigma$. If moreover $\p E$ is a $C^{1,\gamma}$ surface near $0$ then
 $$\int_{\R^n} \frac{\chi_{E^c}-\chi_E}{|x|^{n+\sigma}} \, dx \le |\nabla u^+(0)|^2-|\nabla u^-(0)|^2.$$
\end{lemma}

\begin{proof} We follow closely the proof of Theorem 5.1 of \cite{crs}.

After a dilation we may assume that $E^c$ contains $B_2(-2e_n)$.
Fix $\delta>0$ small, and $\eps \ll \delta$.
Let $T$ be the radial reflection with respect to the sphere
$\partial B_{1+\eps}(-e_n)$

We define the sets:
$$A^-:=B_{1+\eps}(-e_n) \cap E, \quad A^+:=T(A^-) \cap E, \quad A:= A^- \cup A^+.$$
and let
$$F:=T(B_\delta \cap (E \setminus A)).$$
It is easy to check that
$F \subset E^c \cap B_\delta.$

Let $\bar u$ be the perturbation of $u$ which has as positive set $E \setminus A$ as in the proof of Lemma \ref{d88d99.99}. First we estimate the right hand side in the energy inequality \eqref{ei}. Let $\tilde u$ be the perturbation of $u$ which has as positive set $E \setminus A^-$. We use Lemmata \ref{Dbe2.bis}
and~\ref{Dbe2} and we obtain
\begin{align*}
\int_{B_1}|\nabla \bar u|^2-|\nabla u|^2 dx&=\int_{B_1}|\nabla \bar u|^2-|\nabla \tilde u|^2 dx + \int_{B_1}|\nabla \tilde u|^2-|\nabla u|^2 dx  \\
& \le \int_{B_1}|\nabla \bar u^+|^2-|\nabla \tilde u^+|^2 dx + \int_{B_1}|\nabla \tilde u^+|^2-|\nabla u^+|^2 dx\\
 & \le C |A| \, \, \, \|u^+\|_{L^\infty(B_1)}^2.
\end{align*}
Notice that $$\tilde u^+=u^+_{E^c \cup A^-}, \quad \bar u^+= u^+_{E^c \cup A^- \cup T(A^-)}$$
and, by Theorem \ref{H}, $T(A^-)$ satisfies the uniform density property of Lemma \ref{Dbe2}.

Now we consider the left hand side of the energy inequality \eqref{ei}:
$$\Per(E,B_1) -\Per (E \setminus A, B_1)= L(A,E^c)-L(A,\mathcal{C}(E \setminus A)= $$
$$ [L(A, E^c \setminus B_\delta) -L(A,E \setminus B_\delta)]+
[L(A, F)-L(A, T(F))] + L(A,(E^c \cap B_\delta) \setminus F) $$
$$:= I_1+I_2 +I_3 \ge I_1+I_2.$$
We estimate $I_1$ and $I_2$ as in \cite{crs}, and we conclude that
$$\left | \frac{1}{|A|}I_1 -\int_{\mathbb{R}^n \setminus B_\delta}
\frac{\chi_{E^c}-\chi_{E}}{|x|^{n+s}}dx \right | \le C \eps^{1/2} \delta^{-1-s}$$
and
$$I_2 \ge-C \delta^{1-s} |A|- C \eps L(A^-, F).$$
It remains to show that for all small $\eps$
\begin{equation}\label{50}
L(A^-,F) \le C L(A^-, B^c_{1+\eps}(-e_n))
\end{equation}
since then, as in Lemma 5.2 of \cite{crs}, there exists a sequence of $\eps \to 0$ such that
$$\eps L(A^-, F) \le C \eps^\eta |A^-|,$$
and our result follows.

 We prove \eqref{50} by writing the energy inequality for $\tilde u$ defined above. We have $L(A^-,F) \le L(A^-, E^c)$ and
 \begin{align*}
 L(A^-, E^c) & \le L(A^-, E \setminus A^-) + \int_{B_1}|\nabla \tilde u^+|^2 -|\nabla u^+|^2 dx\\
 & \le L(A^-, B^c_{1+\eps}(-e_n)) + C |A^-| \, \|u\|^2_{L^\infty(B_1)} \\
 & \le 2L(A^-, B^c_{1+\eps}(-e_n)).
\end{align*}
where the last inequality holds for all small $\eps$.

In the case when $\p E$ is a $C^{1,\gamma}$ surface near $0$ we can estimate the change in the Dirichlet integral by Lemma \ref{c1g} and obtain the second part of our conclusion. \end{proof}

With the results already obtained,
Theorem \ref{flat} now
follows easily from the improvement of flatness property of $\p E$:

\begin{prop}{\label{min_flat}}
Assume $(u,E)$ is a minimal pair in $B_1$ and
fix $0<\alpha<s$. There exists  $k_0$ depending on $s$, $n$ and
$\alpha$ such that if
$$0 \in \p E, \quad  \quad \|u\|_{L^\infty(B_1)} \le 1,\quad \mbox{and for all balls $B_{2^{-k}}$ with $0 \le k \le k_0$ we have}$$
\begin{equation}\label{k}
\{x \cdot e_k > 2^{-k(\alpha +1)}\} \subset E \subset \{x \cdot e_k > -2^{-k(\alpha +1)}\}, \quad \quad |e_k|=1,
\end{equation}
then there exist vectors $e_k$ for all $k \in \mathbb{N}$ for
which the inclusion above remains valid.
\end{prop}

The proof now follows closely Theorem 6.8 in \cite{crs}. We sketch it below.

Assume \eqref{k} holds for some large $k \ge k_0$. Then by comparison principle we find that
$$u^\pm \le C r, \quad \mbox{in $B_r$ for all $r\ge 2^{-k}$}. $$
for some $C$ depending on $n$ and $\alpha$.

Rescaling by a factor $2^{k}$ the pair $(u,E)$, the situation above can be described as follows: if for all $l$ with $ 0\le l \le k$
$$\|u\|_{L^\infty(B_{2^l})} \le 2^l 2^{-(\sigma k)/2},$$
$$\partial E \cap B_{2^l} \subset \{|x \cdot e_l| \leq 2^l2^{\alpha(l-k)}\},\quad |e_l|=1$$
then the inclusion holds also for $l=-1$, i.e.
\begin{equation}\label{51}
\partial E \cap B_{1/2} \subset \{|x \cdot e_{-1}| \leq 2^{-1} 2^{-\alpha(k+1)}\}.
\end{equation}
For some fixed $l$ we see that $\p E
\cap B_{2^l}$ has $C(l) 2^{-\alpha k}$ flatness, and $u$ is bounded by $C(l) 2^{-(\sigma k)/2}$ in $B_{2^l}$.

First we give a rough Harnack inequality that provides compactness for a sequence of blow-ups.

\begin{lemma}
Assume that for some large $k$, ($k>k_1$)
$$\partial E \cap B_1 \subset \{|x_n| \leq a:=2^{-k\alpha}\}, \quad \quad \|u\|_{L^\infty(B_1)} \le a^{\sigma/(2 \alpha)}$$
and
$$\partial E \cap B_{2^l} \subset \{|x \cdot e_l| \leq a 2^{l(1+\alpha)}\}, \quad l=0,1,\ldots, k.$$
Then either
$$\partial E \cap B_\delta \subset \left\{\frac{x_n}{a} \leq 1-\delta^2\right\} \quad \mbox{or} \quad \partial E \cap B_\delta \subset \left\{\frac{x_n}{a} \geq -1+\delta^2\right\},$$
for $\delta$ small, depending on $\sigma$, $n$, $\alpha$, ($\alpha < \sigma$).

\end{lemma}

\begin{proof}
The proof is the same as Lemma 6.9 in \cite{crs}. The only difference is that at the contact point $y$ between the paraboloid $P$ and $\p E$ the quantity
\begin{equation}\label{q}
\frac 1 a  \, \int_{\R^n} \frac{\chi_{E^c}-\chi_E}{|x-y|^2} dx
\end{equation}
is not bounded above by $0$, instead by Lemma \ref{ELvs}, it is bounded by
$$ \frac 1a C \|u\|_{L^\infty(B_1)}^2 \le C a^{(\sigma/\alpha)-1} \to 0 \quad \mbox{as $a\to 0$},$$
and all the arguments apply as before.
\end{proof}

\noindent{\it Completion of the proof of Proposition \ref{min_flat}.}
As $k$ becomes much larger than $k_1$, we can apply Harnack inequality several times as in \cite{crs}. This gives compactness of the sets
$$\p E^*:=\left\{(x',\frac{x_n}{a}) {\mbox{ s.t. }} x \in \p E \right\},$$
as $a \to 0$. Precisely, we consider pairs $(u,E)$ that are minimal in $B_{2^k}$ with $0\in \p E$, for which
$$\partial E \cap B_1 \subset \{|x_n| \leq a:=2^{-k\alpha}\}, \quad \quad \|u\|_{L^\infty(B_1)} \le a^{\sigma/(2 \alpha)}.$$
and for all $0 \le l \le k$
$$\partial E \cap B_{2^l} \subset \{|x \cdot e_l| \leq a 2^{l(1+\alpha)}\},
 \quad \quad \|u\|_{L^\infty(B_{2^l})} \le 2^l a^{\sigma/(2 \alpha)}.$$
and we want to show that \eqref{51} holds.

If $(u_m,E_m)$ is a sequence of pairs as above with $a_m \rightarrow 0$
there exists a subsequence $m_k$ such that
$$\partial E_{m_k}^* \rightarrow (x',\omega(x'))$$
uniformly on compact sets, where $\omega:\R^{n-1} \to \R$ is H\"older
continuous and
$$\omega(0)=0, \quad |\omega| \leq C(1+ |x'|^{1+\alpha}).$$
Moreover, since the quantity in \eqref{q} tends to $0$, the proof of Lemma 6.11 of \cite{crs} works as before, thus
$$\triangle^\frac{\sigma+1}{2}w=0 \quad \mbox{in $\R^{n-1}$}.$$
This shows that $\omega$ is a linear function and therefore \eqref{51} holds for all large $m$.~\qed

\section{A monotonicity formula}\label{M:S:}

The goal of this
section is to establish
a Weiss-type monotonicity formula for minimizing pairs $(u,E)$,
that is different from the Alt-Caffarelli-Friedman monotonicity formula used
in Lemma \ref{oa}.
For this scope,
we first introduce the localized energy for the $\sigma$-perimeter by using the extension problem in one more dimension as in \cite{crs}.
With a measurable set $E\subset \R^n$  we associate a function $U(x,z)$ defined in $\R^{n+1}_+$ as
$$ U(\cdot,z):=(\chi_E-\chi_{E^c})*P(\cdot,z),\quad \mbox{with}\quad P(x,z):=\tilde c_{n,\sigma} \frac{z^\sigma}{(|x|^2+z^2)^{(n+\sigma)/2}},$$
where $\tilde c_{n,\sigma}$ is a normalizing constant depending on $n$ and $\sigma$.

For a bounded Lipschitz domain $\Omega\subset \R^{n+1}$ we denote by
$$\Omega_0:=\Omega \cap \{z=0\}\subset \R^n, \quad \quad \Omega_+
:= \Omega \cap \{z >0\},$$
and denote the extended variables as
$$X:=(x,z)\in \R^{n+1}_+, \quad \quad \mathcal B^+_r:=\{|X| <r\}.$$

The relation between the $\sigma$-perimeter and its extension is given by Lemma 7.2 in \cite{crs}. Precisely, let $E$ be a set with $\Per(E,B_r)<\infty$ and $U$ its extension, and let $F$ be a set which coincides with $E$ outside a compact set included in $B_r$. Then
$$\Per(F,B_r)-\Per(E,B_r)=c_{n,\sigma} \inf_{\Omega, V} \int_{\Omega^+} z^{1-\sigma}(|\nabla V|^2 -
|\nabla U|^2)dX.$$
Here the infimum is taken over all bounded Lipschitz sets with $\Omega_0 \subset B_r$ and all functions $V$ that agree with $U$ near $\p \Om$ and whose trace on $\{z=0\}$ is given by $\chi_F-\chi_{F^c}$. The constant $c_{n,\sigma}>0$ above is a normalizing constant. As a consequence we obtain the following characterization of minimizing pairs $(u,E)$ using the extension $U$ of $E$.

\begin{prop}\label{EXT:TH}
The pair $(u,E)$ is minimizing in $B_r$ if and only if
\begin{eqnarray*} && \int_{B_r}|\nabla u|^2\,dx+
c_{n,\sigma}\int_{\Omega^+} z^{1-\sigma} |\nabla U |^2\,dX\\ &&\qquad\le
\int_{B_r}|\nabla v|^2\,dx+
c_{n,\sigma}\int_{\Omega^+} z^{1-\sigma} |\nabla V|^2\,dX\end{eqnarray*}
for any bounded Lipschitz domain $\Omega$ with $\Omega_0\subset B_r$ and any functions $v$, $V$ that satisfy

1) $V=U$ in a neighborhood of $\p \Omega$,

2) the trace of $V$ on $\{z=0\}$ is $\chi_F-\chi_{F^c}$ for some set $F\subset \R^n$,

3) $v=u$ near $\p B_r$, and $v\ge 0$ a.e. in $F$, $v \le 0$ a.e. in $F^c$.
\end{prop}

Now we present a Weiss-type monotonicity formula for minimizing pairs $(u,E)$.

\begin{theorem}\label{MONO:F}
Let $(u,E)$ be a minimizing pair in $B_\rho$. Then
\begin{align*}
\Phi_u(r):=r^{\sigma-n}&\left(\int_{B_r}|\nabla u|^2\,dx+c_{n,\sigma}\int_{\mathcal B_r^+}
z^{1-\sigma} |\nabla U|^2\,dX\right)\\
&-\left(1-\frac \sigma 2\right)r^{\sigma-n-1}\int_{\p B_r} u^2 \,d{\mathcal{H}}^{n-1}
\end{align*}
is increasing in $r \in (0,\rho)$.

Moreover, $\Phi_u$ is constant if and only if $u$ is homogeneous
of degree $1-\frac\sigma2$
and $U$ is homogeneous of degree $0$.
\end{theorem}

\begin{proof} The proof is a suitable modification
of the one of Theorem 8.1 in \cite{crs}.
We notice that $\Phi_u$ possesses the natural scaling
\begin{equation*}
\Phi_u(rs)=\Phi_{u_r}(s),
\end{equation*}
where $(u_r,E_r)$ is the rescaling given in \eqref{res}.

We prove that
$$ \frac{d}{dr}\Phi(u,U,r)\ge0 {\mbox{ for a.e. $r$}}.$$
By scaling it suffices to consider the case when $r=1$ and $r$ is a ``regular" radius for $|\nabla u|^2 dx$, $z^{1-\sigma}|\nabla U|^2dxdz$ and $E$.
We use the short notation $\Phi(r)$ for $\Phi_u(r)$ and write
$$\Phi(r)=G(r)-H(r),$$
with
\begin{align*}
G(r)&:=
r^{\sigma-n}\left(\int_{B_r}|\nabla u|^2\,dx+c_{n,\sigma}\int_{B_r^+}
z^{1-\sigma} |\nabla U|^2\,dX\right)\\
H(r)&:=\left(1-\frac \sigma 2 \right) r^{\sigma-n-1}\int_{\p B_r} u^2 \,d{\mathcal{H}}^{n-1}.
\end{align*}
Below we use the minimality to obtain a bound for $G'(1)$. We denote as usual $u_\nu$ and $u_\tau$ for the
normal and tangential gradient of $u$ on $\p B_r$. Let $\epsilon >0$ be small. We compute $G(1)$
by writing the integrals in $ B_{1-\eps}$ and $ B_1\setminus B_{1-\eps}$:
\begin{align*}\label{e.1}
 G_u(1) =&
\int_{B_{1-\eps}}|\nabla u|^2\,dx+\eps\int_{\partial B_1}
|\nabla u|^2 \,d{\mathcal{H}}^{n-1}\\
& +c_{n,s} \left( \int_{\mathcal B^+_{1-\eps}}
z^{1-\sigma}|\nabla U|^2\,dx\,dz + \eps\int_{\partial \mathcal B_1^+}z^{1-\sigma}|\nabla U|^2\,d{\mathcal{H}}^n \right) +o(\eps) \\
=&
(1-\eps)^{n-\sigma}
G(1-\eps)+\eps\int_{\partial B_1}
|u_\tau|^2+|u_\nu|^2\,d{\mathcal{H}}^{n-1}\\
&\quad \quad \quad +
\eps \, \, c_{n,\sigma }\int_{\partial B_1^+}z^{1-\sigma}(|U_\tau|^2+|U_\nu|^2)\,d{\mathcal{H}}^n
+o(\eps).
\end{align*}
We now consider a competitor $(u^\eps,U^\eps)$ for $(u,U)$ defined as
\begin{eqnarray*}
&&u^\eps(x):=\left\{
\begin{matrix}
(1-\eps)^{1-\frac \sigma 2} \, \, u(\frac {x}{1-\eps}) & {\mbox{ if $x\in B_{1-\eps}$,}}\\
\\
|x|^{1-\frac\sigma2)} \, \, u(\frac{x}{|x|})
& {\mbox{ if $x\in B_1\setminus B_{1-\eps}$,}}\\
\\
u(x)& {\mbox{ if $x\in B_1^c$,}}\\
\end{matrix}
\right.
\end{eqnarray*}
and
\begin{eqnarray*}
U^\eps(X):=\left\{
\begin{matrix}
U(\frac{X}{1-\eps}) & {\mbox{ if $x\in \mathcal B_{1-\eps}^+$,}}\\
\\
U(\frac{X}{|X|})
& \quad {\mbox{ if $x\in \mathcal B_1^+\setminus \mathcal B_{1-\eps}^+$,}}\\
\\
U(X)& {\mbox{ if $|X| \ge 1.$}}\\
\end{matrix}
\right.
\end{eqnarray*}
From Proposition \ref{EXT:TH} we obtain
$$G_u(1) \le G_{u^\eps}(1).$$
We compute $G_{u^\eps}(1)$ noticing that $u^\eps$ in $B_{1-\eps}$ coincides with the rescaling $u_{1/(1-\eps)}$ hence
\begin{align*}
G_{u^\eps}(1)&=(1-\eps)^{n-\sigma}G_{u_{1-\eps}}(1-\eps) + \eps \, \,  c_{n,\sigma}\int_{\p \mathcal B_1^+} |U_\tau|^2 d \mathcal H^n\\
& \quad + \eps \int_{\p B_1} \left(|u_\tau|^2 + \Big(1-\frac{\sigma}{2}\Big)^2 u^2\right) \,  d \mathcal H^{n-1} +o(\eps).
\end{align*}
By scaling, the first term in the sum above equals $(1-\eps)^{n-\sigma} G_u(1)$. Plugging $G_u(1)$ and $G_{u^\eps}(1)$ in the inequality above gives
 \begin{align*}
 G_u(1) \ge G_u(1-\eps)&+
\eps \int_{\partial B_1} |u_\nu|^2-\left(1-\frac\sigma2\right)^2 u^2\,d{\mathcal{H}}^{n-1}\\
&+
\eps \, \, c_{n,\sigma}\int_{\partial \mathcal B_1^+}z^{1-\sigma}|U_\nu|^2\,d{\mathcal{H}}^n
+o(\eps),
 \end{align*}
 hence
\begin{equation*}\label{e.2}
G'(1)\ge\int_{\partial B_1} |u_\nu|^2-\left(1-\frac\sigma 2\right)^2 u^2\,d{\mathcal{H}}^{n-1}
+c_{n,\sigma}\int_{\partial \mathcal B_1^+}z^{1-\sigma}|U_\nu|^2\,d{\mathcal{H}}^n
.\end{equation*}
On the other hand,
\begin{equation*}\label{e.3}
H'(1)=\left(1- \frac\sigma 2 \right) \int_{\partial B_1} 2 u u_\nu + (\sigma-2) u^2\,d{\mathcal{H}}^{n-1}.
\end{equation*}
and we conclude that
$$\Phi'(1)\ge
\int_{\partial B_1} \left(u_\nu-\Big(1-\frac\sigma2\Big) u\right)^2\,d{\mathcal{H}}^{n-1}
+c_{n,\sigma}\int_{\partial \mathcal B_1^+}z^{1-\sigma}|U_\nu|^2\,d{\mathcal{H}}^n,$$
and the conclusion follows.
\end{proof}

The monotonicity formula allows us to characterize the blow-up limit of a sequence of rescalings $(u_r,E_r)$. First we need to show that minimizing pairs remain closed under limits.

\begin{prop}{\label{cl_limit}} Assume $(u_m,E_m)$ are minimizing pairs in $B_2$ and $$u_m \to u \quad \mbox{in $L^2(B_2)$}, \quad \mbox{and} \quad
E_m \to E \quad \mbox{in $L^1_{\rm loc}(\R^n)$}.$$
Then $(u,E)$ is a minimizing pair in $B_1$ and $u_m \to u$ in $H^1(B_1)$ and
$$\Per(E_m,B_1) \to \Per(E,B_1).$$
\end{prop}

\begin{proof} First we show that $u_m \to u$ in $H^1(B_1)$. Since $\nabla u_m \rightharpoonup \nabla u$ weakly in $L^2$ it suffices to show that $$\int_{B_1}|\nabla u_m|^2 dx \to \int_{B_1}|\nabla u|^2 dx.$$
Indeed, since $u_m$ and $u$ are continuous functions which are harmonic in their positive and negative sets we have
$$\triangle u^2=2|\nabla u|^2, \quad \triangle u_m^2=2|\nabla u_m|^2,$$
and the limit above follows since $u_m^2 \to u^2$ in $L^1$.

Let $(v,F)$ be a compact perturbation for $(u,E)$ in $B_1$. Precisely, assume $F=E$ and $v=u$ outside a compact set of $B_1$, and $v \ge 0$ a.e. in $F$, $v \le 0$ a.e. in $F^c$. Let
$$w_m^+=\min\{u_m^+,u^+\}$$ and define $v_m^+$ such that $v_m^+=v^+$ in $B_{1-2\eps}$, $v_m^+=w_m^+$ in the annulus $B_{1+\eps} \setminus B_{1-\eps}$ and $v_m^+=u_m^+$ outside $B_{1+2\eps}$. In $B_1 \setminus B_{1-2\eps}$ we define $v_m^+$ as an interpolation between $v^+$ and $w_m^+$ i.e.
$$v_m^+=\eta v^+ + (1-\eta) w_m^+,$$
with $\eta$ a cutoff function with $\eta=1$ in $B_{1-2\eps}$ and $\eta=0$ outside $B_{1-\eps}$.
Similarly, in $B_{1+2\eps} \setminus B_1$ we let $v_m^+$ to be an interpolation between $u_m^+$ and $w_m^+$.

We define $v_m^-$ similarly. We have
$$v_m \ge 0 \quad \mbox{a.e. in $F_m$}, \quad  v_m \le 0 \quad \mbox{a.e. in $F_m^c$, with}$$
$$F_m:= (F \cap B_1) \cup (E_n \setminus B_1),$$
thus $(v_m,F_m)$ is a compact perturbation of $(u_m,E_m)$. {F}rom the minimality of $(u_m,E_m)$ (see \eqref{J}) we find
$$J_{B_2}(u_m) \le J_{B_2}(v_m).$$
By construction,
$$\int_{B_2}|\nabla v_m|^2-|\nabla u_m|^2 dx \le \int_{B_1}|\nabla v|^2 -|\nabla u_m|^2 dx + c_m(\eps),$$
with $$c_m(\eps):= C \eps^{-2} \int_{B_2}(u_m-u)^2 dx + C \int_{B_{1+2\eps}-B_{1-2\eps}}|\nabla u|^2+|\nabla u_m|^2dx.$$
Notice also that
$$\Per(F_m,B_2) - \Per (E_m,B_2) \le \Per(F,B_1)-\Per(E_m,B_1) + b_m,$$ with $$b_m:=L(B_1, (E_m \Delta E) \setminus B_1).$$
Since $E_m \to E$ in $L^1_{\rm loc}(\R^n)$ it follows easily that $b_m \to 0$ (see Theorem 3.3 in \cite{crs}). Using the last two inequalities in the energy inequality and letting first $m \to \infty$ and then $\eps \to 0$ we find
$$\limsup J_{B_1}(u_m) \le J_{B_1}(v).$$
On the other hand from the lower semicontinuity of $J$ we have
$$\liminf J_{B_1}(u_m) \ge J_{B_1}(u).$$
This shows that $(u,E)$ is a minimizing pair and that $J_{B_1}(u_m) \to J_{B_1}(u)$ and our conclusion follows.~\end{proof}

Next we consider the limit of a sequence of rescalings $u_r$, $E_r$, $U_r$ as $r \to 0$,
$$u_r(x)=r^{\frac \sigma 2 -1}u(rx), \quad E_r=r^{-1}E, \quad U_r(X)=U(rX).$$

\begin{prop}[Tangent cone]\label{tc}
Assume $(u,E)$ is a minimizing pair in $B_1$, and $0 \in \p E$. There exists a sequence of $r=r_k \to 0$ such that
$$u_r \to \bar u \quad \mbox{in $L^2_{\rm loc}(\R^n)$}, \quad E_r \to \bar E \quad \mbox{in $L^1_{\rm loc}(\R^n)$}, \quad U_r \to \bar U \quad \mbox{in $L^2_{\rm loc}(\R^n, \ z^{1-\sigma}dX)$}$$
with $\bar u$ homogeneous of degree $1-\frac \sigma 2$, $\bar U$ homogeneous of degree $0$ and $(\bar u, \bar E)$ a minimizing pair in $\R^n$.
\end{prop}

We refer to a minimizing homogeneous pair $(\bar u, \bar E)$ as a {\it minimizing cone}. {F}rom Theorem \ref{H} we see that on compact sets $u_r \to u$ uniformly and $E_r \to \bar E$ in Hausdorff distance.

\begin{proof}
By compactness we can find a sequence such
that $u_r \to \bar u$ and $E_r \to \bar E$ as above. {F}rom Proposition \ref{cl_limit} we have $\Per(E_r) \to \Per(\bar E)$
and, as in Proposition 9.1 in \cite{crs}, this implies the convergence above of $U_r$ to $U$, and
$$\Phi_{u_r}(t) \to \Phi_{\bar u}(t) \quad \mbox{as $r \to 0$}.$$
Then $\Phi_{\bar u}(t)=\Phi_u(0+)$ and the conclusion follows from Theorem \ref{MONO:F}. 
Notice from the definition of $\Phi$ that $\Phi(0+)$ is bounded
since $u \in C^\alpha(B_1)$, with $\alpha=1-\frac\sigma2$, thanks to
Theorem \ref{H}.~\end{proof}

Let $(\bar u, \bar E)$ be a minimizing cone. We define its energy as $\Phi_{\bar u}$ which is a constant
(recall Theorem \ref{MONO:F}). 
{F}rom the homogeneity of $\bar u$ it follows that $$\Phi_{\bar u}=c_{n,\sigma}\int_{B_1} |\nabla \bar U|^2 dX,$$
hence the energy depends only on $\bar E$. 

Since $\bar u^\pm$ are complementary homogeneous harmonic functions in $\bar E$ respectively $\bar E^c$, at least one of them, say $\bar u^-$, has homogeneity greater or equal to 1, thus $\bar u^-=0$. Then $\bar u^+$ is homogeneous of degree $1-\frac\sigma2$ and
$$\int_{\R^n} \frac{\chi_{\bar E^c}-\chi_{\bar E}}{|y-x|^{n+\sigma}} \, dy=|\nabla \bar u^+(x)|^2, \quad \quad \forall x \in \p \bar E,$$
holds in the viscosity sense. Notice that both terms are homogeneous of degree $-\sigma$. 
 
 If $\bar u^+ \equiv 0$ then the study of minimizing cones reduces to the study of $\sigma$-minimal surfaces. This is the case when $\sigma=1$ which was treated in \cite{acks}. Indeed, the homogeneity of a positive harmonic function in a mean-convex cone $E$ which vanishes on $\p E$ cannot be less than 1. This follows since a multiple of the distance function to $\p E$ is superharmonic and is an upper barrier for $\bar u^+$. When $\sigma<1$ it is not clear whether or not there exist minimizing cones with $\bar u \ne 0$ and it seems difficult to relate the $\sigma$-curvature of $\p E$ with the homogeneity of $\bar u^+$.

When $\bar E=\Pi$ is a half-space then $\bar u \equiv 0$ and we call $(0,\Pi)$ a trivial cone. If the blow-up limit $(\bar u, \bar E)$ of a minimizing pair $(u,E)$ is trivial then we say that $0\in \p E$ is a regular point of the free boundary. By Theorem \ref{flat}, $\p E$ is a $C^{1,\gamma}$ surface in a neighborhood of its regular points.
 
 We remark that if $E$ admits an exterior tangent ball at $0 \in \p E$ then $\bar E \subset \Pi$ and $\bar u^+=0$. Then, we use the Euler-Lagrange equation (Lemma \ref{ELvs}) and obtain $E=\Pi$. Thus any point on $\p E$ which admits a tangent ball from $E$ or $E^c$ is a regular point. Therefore the set of regular points is dense in $\p E$.  We summarize these results below.
 
 \begin{prop} Let $(u,E)$ be a minimal pair, $0 \in \p E$, and let $(\bar u, \bar E)$ be its tangent cone as in Proposition \ref{tc}. If $\bar E$ is a half-space (i.e. if $0$ is a regular point) then $\p E$ is a $C^{1,\gamma}$ surface and the free boundary equation \eqref{EL} holds. Moreover, all points on $\p E$ which have a tangent ball from either $E$ or $E^c$ are regular points.
\end{prop}
By a standard argument (see Theorem 9.6 in \cite{crs}), we also obtain that the trivial cone has the least energy amongst all minimizing cones. Precisely if $(\bar u, \bar E)$ is a minimizing cone then
$$\Phi_{\bar u} \ge \Phi_\Pi,$$
and if $\bar E$ is not a half-space then
$$\Phi_{\bar u} \ge \Phi_\Pi+ \delta_0$$
for some $\delta_0>0$ depending only on $n$, $\sigma$.

\section{Proof of Theorem \ref{FB reg}}

In this section we prove Theorem \ref{FB reg} using the dimension reduction argument of Federer. As in Section 10 in \cite{crs}, in order to obtain Theorem \ref{FB reg} it suffices to prove the following two propositions.

\begin{prop}\label{dr}
The pair $(u,E)$ is  minimizing in $\R^n$ if and only if $(u(x), E \times \R)$ is minimizing in $\R^{n+1}$.
\end{prop}

\begin{prop}\label{2d}
In dimension $n=2$, all minimizing cones are the trivial.
\end{prop}

\noindent{\it Proof of Proposition \ref{dr}.} The proof is similar to the one of Theorem 10.1 in \cite{crs}. We just sketch the main difference. The only issue that needs to be discussed is the existence of a perturbation which is admissible when we prove that  $(u,E)$ is minimizing in $\R^n$ if $(u(x), E \times \R)$ is minimizing in $\R^{n+1}$.
 
Precisely let $v(x)$, $V(x,z)$  be admissible functions which coincide with $u$, respectively $U$ say outside $\mathcal B^+_{1/2}$. 
It suffices to construct an admissible pair $w(x,x_{n+1})$ and $W(x,x_{n+1},z)$ in one dimension higher 
i.e. in $\mathcal B_1 \times [0,1]$ such that on the $n$ dimensional slice $x_{n+1}=0$, 
$(w,W)$ coincides with $(u,U)$, and on the slice $x_{n+1}=1$, $(w,W)$ coincides with $(v,V)$. 

For $x_{n+1}\in [0,1/4]$ we define $$W(x,x_{n+1},z)=U(x,z), \quad \mbox{and} \quad w(x,x_{n+1}):=(1-\varphi+ \varphi \eta(x))u(x)$$
with $\varphi=\varphi(x_{n+1})$ a smooth function vanishing for $x_{n+1} \le 0$ and which equals 1 for $x_{n+1}\ge 1/4$. 
The  function $\eta$ above is a cutoff function which vanishes in $B_{1/2}$ and equals 1 outside $B_{3/4}$.

Similarly we construct $W$ and $w$ for $x_{n+1}\in [3/4,1]$, by using the pair $(v,V)$. 

In the interval $x_{n+1}\in [1/4, 3/4]$ we extend $w$ to be constant in the $x_{n+1}$ variable. 
We also extend $W$ to be constant in the annulus $\mathcal B_1^+\setminus \mathcal B_{1/2}^+$.
It remains to construct $W$ in the inner cylinder $\mathcal B_{1/2} \times [1/4,3/4]$. 
Since $w=0$ on the ``bottom" of this cylinder, any choice for $W$ with trace $\pm 1$ on $\{x_{n+1}=0\}$ makes the pair $(w,W)$ admissible. 
Now we can argue precisely as in the proof of the  $\sigma$-minimal surfaces, and the construction for the interpolating $W$ is given in Lemma~10.2 in~\cite{crs}.~\qed

\noindent{\it Proof of Proposition \ref{2d}.}
We follow the methods in \cite{CVcone, CVgen} where the same result was proved for $\sigma$-minimal surfaces. 
We remark that the assumption that $n=2$ is only necessary
at the end of the proof.
We define
$${\mathcal{E}}_r (v,V):= \int_{B_\rho}|\nabla v|^2\,dx + c_{n,\sigma} \int_{\mathcal B_r^+} z^{1-\sigma}|\nabla V(X)|^2\,dX.$$
By Proposition \ref{EXT:TH}, we know that $(u,U)$
minimizes ${\mathcal{E}}$ under domain variations. We consider a diffeomorphism on $\R^{n+1}$ given, for any
$X\in\R^{n+1}_+$ by
\begin{equation}\label{d.1} X \mapsto Y:=X+\varphi(|X|/R)e_1,\end{equation}
where $\varphi\in C^\infty(\R)$, $\varphi=1$ in $[-1/2,1/2]$
and $\varphi=0$ outside $(-3/4,3/4)$, and $R$ is a large parameter.
We define $U^+_R(Y):=U(X)$ and similarly, if we change $e_1$ into $-e_1$ in \eqref{d.1},
we may define $U^-_R$.
The diffeomorphism in \eqref{d.1} restricts to
a diffeomorphism in $\R^n$ just by considering points
of the type $X=(x,0)$, i.e.
$$ y:=x+\varphi(|x|/R)e_1.$$
and we set $u^+_R(y):=u(x)$, and similarly  we define $u^-_R$.
We claim that
\begin{equation}\label{7.8}
{\mathcal{E}}_R(u^+_R,U^+_R)+
{\mathcal{E}}_R(u^-_R,U^-_R)-2{\mathcal{E}}_R(u,U)\le C R^{n-2-\sigma},
\end{equation}
for some $C$ independent of $R$. 
By Proposition \ref{EXT:TH}, the minimality of $(u,U)$ gives 
$${\mathcal{E}}_R(u,U)\le {\mathcal{E}}_R(u^-_R,U^-_R),$$
and the last two inequalities imply
\begin{equation}\label{7.9}
{\mathcal{E}}_R(u_R^+,U_R^+)\le {\mathcal{E}}_R(u,U) + C R^{n-2-\sigma}.
\end{equation}
To prove~\eqref{7.8}, by direct calculations (or see formula~(11) in~\cite{CVcone}) we obtain
\begin{align*}
\Big( |\nabla u^+_R|^2+|\nabla u^-_R|^2 \Big)\,dy &=
2(1+O(1/R^2)\chi_{B_R\setminus B_{R/2}})|\nabla u|^2\,dx\
\\
z^{1-\sigma}\Big( |\nabla U^+_R|^2+|\nabla U^-_R|^2 \Big)\,dY&=
2z^{1-\sigma}(1+O(1/R^2)\chi_{B_R^+\setminus B_{R/2}^+})|\nabla U|^2\,dX.
\end{align*}
We use that $|\nabla u(x)|^2$ and $z^{1-\sigma} |\nabla U(X)|^2$ are homogeneous of degree $-\sigma$ respectively $-1-\sigma$
and obtain 
\begin{align*}
 \int_{B_R}&\Big(|\nabla u^+_R|^2+|\nabla u^-_R|^2 \Big)\,dy -
2\int_{B_R}|\nabla u|^2\,dx \\
&\le CR^{-2}\int_{B_R\setminus B_{R/2}}|\nabla u|^2\,dx \le  CR^{-2} \cdot R^{n-\sigma}
\end{align*}
and
\begin{align*}
 \int_{B_R^+}&
z^{1-\sigma}\Big( |\nabla U^+_R|^2+|\nabla U^-_R|^2 \Big)\,dY-
2\int_{B_R^+}
2z^{1-\sigma}|\nabla U|^2\,dX \\ 
& \le CR^{-2}\int_{B_R^+\setminus B_{R/2}^+}z^{1-\sigma}|\nabla U|^2\,dX \le CR^{-2} \cdot R^{n-\sigma}
\end{align*}
and so the proof of~\eqref{7.8} is complete.

Next we perform an argument similar
to the one
of Theorem~1 of~\cite{CVcone} (the main difference here is that
two functions are involved in the minimization procedure instead of
a single one). For this, we assume now that~$n=2$, we argue
by contradiction and we suppose that $E$ is not a halfplane.
Thus,
there exist~$M>0$ and~$p\in B_M$, say on the $e_2$-axis, such that $p$
lies in the interior of $E$,
and~$p+ e_1$ and $p-e_1$ lie in $E^c$.
Therefore, if~$R$ is sufficiently large we have that
\begin{equation}\label{76}\begin{split}
&{\mbox{$u^+_R(x)=u(x-e_1)$,
for all $x\in B_{2M}$}}\\
&{\mbox{$U^+_R(X)=U(X-e_1)$,
for all $X\in B_{2M}^+$,}}\\
&{\mbox{$u^+_R(x)=U(x)$ for all $x\in\R^2
\setminus B_{R}$, and}}\\
&{\mbox{$U^+_R(X)=U(Y)$ for all $X\in \R^{3}_+\setminus
B_{R}^+$.}}
\end{split}\end{equation}
We define
\begin{equation*}\begin{split}
&v_R(x):=\min \{ u(x), \,u^+_R(x)\}, \quad
w_R(x):=\max \{ u(x), \,u^+_R(x)\},\\
&V_R(X):=\min \{ U(X), \,U^+_R(X)\} \quad  {\mbox{ and }} \quad
W_R(X):=\max \{ U(X), \,U^+_R(X)\}\end{split}\end{equation*}
and~$P:=(p,0)\in\R^3$. {F}rom~\eqref{76} and the trace property of~$U$ we have that
\begin{eqnarray}
&&{\mbox{$U^+_R<W_R=U$
in a neighborhood of~$P$, and}} \label{not1}\\
&&{\mbox{$U<W_R=U^+_R$
in a neighborhood of~$P+e_1$.}}\label{not2}\end{eqnarray}
Moreover
$$ {\mathcal{E}}_R(u,U)\le {\mathcal{E}}_R(v_R,V_R)$$
and
$$
{\mathcal{E}}_R(v_R,V_R)+{\mathcal{E}}_R(w_R,W_R)=
{\mathcal{E}}_R(u,U)+{\mathcal{E}}_R(u^+_R,U^+_R),$$
therefore
\begin{equation}\label{56}\begin{split}
{\mathcal{E}}_R(w_R,W_R)\le {\mathcal{E}}_R(u^+_R,U^+_R).
\end{split}\end{equation}
Now we observe that
\begin{equation*}\label{NOT}
{\mbox{ $(w_R,W_R)$ is not a minimizer for~${\mathcal{E}}_{2M}$}}\end{equation*}
with respect to compact perturbations
in~$B_{2M}\times \mathcal B_{2M}^+$.
Otherwise~$W_R$ would be a minimizer too: then
the fact that~$U\le W_R$, \eqref{not1} and the strong maximum principle
would give that~$U=W_R$ in~$\mathcal B^+_{2M}$, but this would be
in contradiction with~\eqref{not2}.
Thus there exists~$\delta>0$ and
a competitor
\begin{equation*}\label{098}{\mbox{$(u_*,U_*)$ that coincides with $(w_R,W_R)$ outside~$B_{2M}\times \mathcal B_{2M}^+$}}
\end{equation*}
(with $u_*=w_R$) and such that
\begin{equation*}\label{92}
{\mathcal{E}}_{2M}(u_*,U_*)+\delta \le {\mathcal{E}}_{2M}(w_R,W_R).
\end{equation*}
Here~$\delta>0$ is independent of~$R$
since~$(w_R,W_R)$
does not depend on $R$ when
restricted to~$B_{2M}\times \mathcal B_{2M}^+$
(recall~\eqref{76}). We
conclude that
\begin{equation*}
{\mathcal{E}}_R (u_*,U_*)+\delta \le {\mathcal{E}}_R(w_R,W_R).\end{equation*}
Combining this with \eqref{7.9} and~\eqref{56} we obtain
$$ {\mathcal{E}}_R (u_*,U_*)+\delta\le {\mathcal{E}}_R(w_R,W_R)\le {\mathcal{E}}_R(u^+_R,U^+_R)
\le {\mathcal{E}}_R(u,U)+CR^{-\sigma}.$$
If~$R$ is large enough
we obtain that~${\mathcal{E}}_R (u_*,U^*) < {\mathcal{E}}_R(u,U)$,
which contradicts the minimality of~$(u,U)$
and completes the proof of Proposition~\ref{2d}.~\qed

\section{Proofs of Lemmas \ref{Dbe2.bis} - \ref{c1g}}

In this section we estimate the difference in the Dirichlet energies of the harmonic replacements in two different sets $E$ and $E\setminus A$, with $A\subset B_{3/4}$. We assume that $\varphi \in H^1(B_1) \cap L^\infty(B_1)$, $\varphi \ge 0$, and let
$$w:=\varphi_{E^c}, \quad \quad v:=\varphi_{E^c \cup A}.$$
Here above, we used
the notation for the harmonic replacements
of $\varphi$ that vanish in $E^c$ and $E^c\cup A$, as introduced in Definition \ref{H.R}.
We remark that the existence of $v$ follows from the existence of $w$. Indeed, given $w$ we can easily find an explicit test function with finite energy which vanishes in $E^c \cup B_{3/4}$, for example a function of the form $w(1-\eta)$ with $\eta$ a cutoff function.

Since $w$ minimizes the Dirichlet energy among all functions which are fixed in $E^c$ and have prescribed values on $\p B_1$ we find
\begin{equation}\label{eleq}
\int_{B_1}\nabla w \cdot \nabla \psi \, dx=0, \quad \quad \quad \forall \psi \in H_0^1(B_1) \quad \mbox{with $\psi =0$ a.e. in $E^c$,}
\end{equation}
and therefore
\begin{equation}\label{eleq1}
\int_{B_1} |\nabla (w-\psi)|^2-|\nabla w|^2 \, \, dx=\int_{B_1}|\nabla \psi|^2 \, dx.
\end{equation}

By definition, $v$ minimizes the Dirichlet energy among all functions which equal $w$ on $\p B_1$, and are $0$ a.e. in $E^c\cup A$. We may relax this last condition to functions that are equal to $0$ a.e. in $E^c$ and are nonpositive in $A$, since then we can truncate them wherever they are negative. This and \eqref{eleq1} show that
\begin{equation}\label{dec}
\int_{B_1} |\nabla v|^2-|\nabla w|^2 \, \, dx= \inf_{\psi \in \mathcal A} \int_{B_1} |\nabla \psi|^2 dx
\end{equation}
where
$$\mathcal A:=\{\psi \in H_0^1(B_1), \quad \psi=0 \quad \mbox{a.e. in $E^c$}, \quad  \psi \ge w \quad \mbox{a.e. in $A$}\}.$$
We use this characterization and show that the difference between the energies of $v$ and $w$ depends monotonically on $\varphi$, $E$ and $A$. Precisely, for $i=\{1,2\}$ let $w_i$, $v_i$ be the corresponding functions for $\varphi_i$, $E_i$, $A_i$.

 \begin{lemma}\label{comp}
Assume $$\varphi_1 \le \varphi_2, \quad E_1 \subset E_2, \quad  A_1 \subset A_2.$$ Then
$$ \int_{B_1} |\nabla v_1|^2-|\nabla w_1|^2 \, \, dx\le \int_{B_1} |\nabla v_2|^2-|\nabla w_2|^2 \, \, dx. $$
 \end{lemma}

\begin{proof}
Let $\bar v_2$ minimize
the Dirichlet integral in $B_1$ among all the functions
that equal $v_2$ a.e. in $E_1^c$ and $\bar v_2-v_2 \in H_0^1(B_1)$. Notice that $\bar v_2$ is well defined since $v_2$ is a test function with finite energy, so the minimizer exists by direct methods.
As in \eqref{eleq} and \eqref{eleq1} above, we find
$$\int_{B_1} |\nabla v_2|^2-|\nabla \bar v_2|^2 \, \, dx=\int_{B_1} |\nabla (\bar v_2- v_2)|^2 \, \, dx.$$

Since $\bar v_2=v_2=0$ a.e. in $E_2^c \subset E_1^c$, and $\bar v_2= w_2$ on $\p B_1$ we find from the definition of $w_2$ that
$$
\int_B |\nabla w_2|^2\,dx\le
\int_B |\nabla \bar v_2|^2\,dx.
$$
hence
$$\int_{B_1} |\nabla (\bar v_2- v_2)|^2 \, \, dx=\int_{B_1} |\nabla v_2|^2-|\nabla \bar v_2|^2 \, \, dx \le \int_{B_1} |\nabla v_2|^2-|\nabla w_2|^2 \, \, dx.$$
Using the characterization in \eqref{dec} for $v_1$, $w_1$ it suffices to show that $\bar v_2-v_2 \in \mathcal A_1$. By construction $\bar v_2-v_2 \in H_0^1(B_1)$, $\bar v_2 - v_2 =0$ a.e. in $E_1^c$ and $\bar v_2-v_2=\bar v_2$ a.e. in $A_1 \subset A_2$. It remains to check that $\bar v_2 \ge w_1$ which follows by maximum principle.

Indeed, let $h:= (w_1-\bar v_2)^+$. We have $h=0$ a.e. in $E_1^c$ and also $h \in H_0^1(B_1)$ since $\varphi_1 \le \varphi_2$. {F}rom the definitions of $w_1$, $\bar v_2$ (see \eqref{eleq}) we obtain
$$ \int_{B_1} \nabla w_1 \cdot\nabla h\,dx=0, \quad \quad \int_{B_1} \nabla \bar v_2\cdot\nabla h\,dx=0.$$
Then
$$ \int_{B_1} |\nabla (w_1- \bar v_2)^+|^2\,dx
=\int_{B_1} \nabla (w_1-\bar v_2)\cdot\nabla h \,dx=0,$$
and the desired inequality $w_1 \le \bar v_2$ is proved.
\end{proof}

\noindent{\it Proof of Lemma \ref{Dbe2.bis}.} After dividing $w$ and $v$ by an appropriate constant, we may assume that $\|w\|_{L^\infty(B_1)}=1$. Then by Lemma  \ref{comp} it suffices to prove our bound in the case when $\varphi =1$, $B_1 \setminus B_\rho \subset E$ and $A= B_\rho \cap E$. In this case
$$v=c(\rho^{2-n}-|x|^{2-n})^+$$ for an appropriate $c$, and using symmetric rearrangement we see that the Dirichlet integral of $w$ is minimized whenever $w$ and the set $A$ are radial. Therefore we need to prove the lemma only in the case when $E=B_r^c$, $A=B_\rho \setminus B_r$, for some $r \le \rho$. We have
$$
\int_{B_1}|\nabla v|^2-|\nabla w|^2 \, dx  = \int_{B_1} |\nabla (w-v)|^2 dx = \int_{B_1 \setminus B_r} (w-v) \triangle (v-w).$$
Using that in $B_\rho \setminus \overline B_r$
$$\triangle (v-w)=\triangle v = v_{\nu} d \mathcal H^{n-1}|_{\p B_\rho},$$
and that $w-v=w \le C r$ on $\p B_\rho$ we find
$$\int_{B_1}|\nabla v|^2-|\nabla w|^2 \, dx  \le  C r \le C |A|,$$
and the lemma is proved.~\qed

\noindent{\it Proof of Lemma \ref{Dbe2}.} Assume that $\|w\|_{L^\infty(B_1)}=1$ and as before, by Lemma \ref{comp}, it suffices to obtain the bound in the case when $\varphi=1$ and $E=B_{1/2}$. Then $$w:=c(2^{n-2}-|x|^{2-n})^+$$ for an appropriate $c$, and let $$ \bar v:= \min\{w, C_0 d_A\},$$
where $d_A$ represents the distance to the closed set $A$, and $C_0$ is a large constant depending only on $n$. Notice that by construction $\bar v-\varphi \in H^1_0(B_1)$, $\bar v=0$ in $A$ and $\bar v$ has bounded Lipschitz norm. Then
$$\int_{B_1}|\nabla v|^2-|\nabla w|^2 \, dx  \le \int_{B_1}|\nabla \bar v|^2-|\nabla w|^2 \, dx  \le C |S|,$$
where $S:=\{\bar v < w\}$.
It remains to show that $|S| \le C(\beta) |A|$ which follows the uniform density property of $A$.

By choosing $C_0$ sufficiently large we have
$$ S \subset \{C_0 d_A < w  \} \subset \{ 6 d_A  < d_{\p B_{1/2}} \}.$$ Thus if $x \in S$ and $y \in \p A$ is the closest point to $x$ then it easily follows that
 $$x \in B_{d_y/5} (y)  \quad \quad \mbox{ with} \quad d_y:= d_{\p B_{1/2}}(y).$$
Hence by Vitali's lemma we can find a collection of disjoint balls $B_{d_{y_i}/5}(y_i)$ such that$$S \subset \, \bigcup_i \, B_{d_{y_i}}(y_i).$$  Thus, by adding the inequalities  $$|A \cap B_{d_{y_i}/5}(y_i)| \ge c(\beta) |B_{d_{y_i}}(y_i)  |$$
we obtain that $|A| \ge c(\beta) |S|$.~\qed

For the proof of Lemma \ref{c1g} we first need a regularization result for the maximum of two $C^{1,\gamma}$ functions, $\gamma \in (0,1)$. In the next lemma we smooth out the ``corners" of the graph of the positive part of a $C^{1,\gamma}$ function without increasing its area too much.

\begin{lemma}\label{env}
Assume $h: \overline \Om \to \R^+$ is a $C^{1,\gamma}$ function that satisfies $\{ h>0 \}=\Om$, $h=0$ on $\p \Om$, and for any $z \in \overline \Om$ there exists a linear function $l_{z}$ (its tangent plane) such that
$$|h-l_{z}| \le \eps |x-z|^{1+\gamma}, \quad \quad \forall x\in \overline \Om, $$
for some $\eps>0$ small.
Let $$K:=\{ z \in \overline \Om{\mbox{ s.t. }} l_{z}+|x-z|^{1+\gamma} \ge 0 \quad \mbox{in $\R^n$}\}$$
and denote by $$ h^*(x):=\inf_{z \in K} \left (l_{z}+|x-z|^{1+\gamma} \right).$$
Then
$$\int_{\overline \Om} h^* \, dx \le (1+\eps^\sigma) \int_K h  \, dx $$
with $\sigma > 0$ depending on $n$ and $\gamma$.
\end{lemma}

 Clearly if we replace $|x-z|^{1+\gamma}$ by $m|x-z|^{1+\gamma}$ the conclusion still holds since the problem remains invariant under multiplication by a constant $m$.
The function $h^*$ can be thought as a $C^{1,\gamma}$ upper envelope of norm $\|\nabla h\|_{C^\gamma}/\eps$ of the function $h$ (extended by $0$ in the whole $\R^n$).

By construction $h^* \ge h$ in $\overline \Om$, $h=h^*$ in $K$, and at any point $z\in K$ the graph of $h$ is tangent by below to the $C^{1,\gamma}$ function $l_z+|x-z|^{1+\gamma} \ge 0$.

\begin{proof} Notice that
$$z \in K \quad \Leftrightarrow \quad h(z) \ge c_0 |\nabla h(z)|^\frac{\gamma+1}{\gamma}, \quad \quad \mbox{with} \quad c_0:=\gamma (\gamma+1)^{-\frac{\gamma+1}{\gamma}}.$$

We show that for any $y \in \overline \Om \setminus K$ there exists $d_y>0$ such that
\begin{equation}\label{int1}
\int_{(\overline \Om \setminus K)\cap B_{d_y}(y)} h^* \, dx \le \eps^\sigma \int_{B_{d_y/5}(y) \cap K} h \, dx.
\end{equation}
Then, by Vitali lemma, we cover $\overline \Om \setminus K$ with a collection of balls $B_{d_{y_i}}(y_i)$ with $B_{d_{y_i}/5}(y_i)$ disjoint and we obtain the desired claim by summing \eqref {int1} for all $y_i$.

Our hypotheses and \eqref{int1} remain invariant under the scaling $$h_\lambda(x)=\lambda^{1+ \gamma} h( x/\lambda ),$$ thus we may assume for simplicity that $y=0$ and $\nabla h(0)=e_n$. Since $0 \notin K$ we have $h(0) \in [0,c_0)$, and by our hypothesis $$|h(x)-(h(0)+ x_n)| \le \eps|x|^{1+\gamma},$$
hence
$$|h(x)-(h(0)+x_n)| \le \eps^{1/2} \quad \mbox{if} \quad |x| \le 2 d_0:= \eps^{-\frac{1}{2(\gamma+1)}}.$$
This implies that for some $C_0$ sufficiently large,
$$ \overline \Om \cap B_{d_0} \subset \{ x_n \ge - C_0\},$$
$$|\nabla h| \le 2, \quad h \ge c_0 2^\frac{\gamma+1}{\gamma}  \quad \mbox {in the set} \quad B_{d_0}\cap\{x_n \ge C_0\}.$$
We obtain $$B_{d_0}\cap\{x_n \ge C_0\} \subset K, \quad \mbox{and} \quad h^* \le C \quad \mbox{in} \quad B_{d_0} \cap \{|x_n| \le C_0\}$$
hence
$$\int_{(\overline \Om \setminus K) \cap B_{d_0}} h^* \, dx\le C d_0^{n-1}, \quad \int_{K \cap B_{d_0/5}} h \, dx \ge c d_0^{n+1},$$
and \eqref{int1} follows.
\end{proof}

Assume for simplicity that $E$ is a set
$$E:=\{x_n \ge g(x')\},$$
where $g$ is a $C^{1,\gamma}$ function and $g(0)=0$, $\nabla_{x'}g(0)=0$.

Let $u \in H^1(E \cap \overline B_1)$, be positive and harmonic in the interior with $u=0$ on $\p E$.
First we state a consequence of $C^{1,\gamma}$ estimates for harmonic functions.

\begin{lemma}\label{c1ga}
Let $F=\{x_n \ge f(x')\}$ be a compact perturbation of $E$ in $B_{1/2}$ and denote by $v$ the harmonic function in $F \cap B_1$ which vanishes on $\p F \cap B_1$ and equals $u$ on $\p B_1$. Assume that
$f$, $g$ are $C^{1,\gamma}$ functions with norm bounded by a constant $M$, $\|u\|_{L^2} \le M$ and also
that $|f-g| \le \eps$. Then
$$\|\nabla u-\nabla v  \|_{L^\infty(E \cap F \cap B_{1/2})}\le C \eps^\frac{\gamma}{1+\gamma}.$$
for some constant $C$ depending on $n$, $\gamma$ and $M$.
\end{lemma}

\begin{proof}
By boundary $C^{1,\gamma}$ estimates $$\|v\|_{C^{1,\gamma}(B_{3/4}\cap F)} \le C \quad  \Rightarrow \quad |u-v| \le C \eps \quad \mbox{on} \quad \p(E \cap F \cap B_1).$$ By maximum principle, the last inequality holds also in the interior of the domain and the conclusion follows since $u-v$ has bounded $C^{1,\gamma}$ norm in $B_{3/4}\cap E \cap F$.
\end{proof}

\noindent{\it Completion of the proof of Lemma \ref{c1g}.}
We estimate the change in the Dirichlet integral for the harmonic replacement of $u$ whenever we perturb $E$ by a small $C^{1,\gamma}$ set $A \subset B_\eps$. We distinguish two cases, when $A$ is interior to $E$ and when $A$ is exterior to $E$.
Assume for simplicity that $|\nabla u(0)|=1$.

\

\noindent{\it Case 1:} The set $A$ is interior to $E$,
\begin{equation}\label{case1}
A=\{g(x') \le x_n < f(x') \} \subset B_\eps,
\end{equation}
for some function $f$ with $C^{1,\gamma}$ norm bounded by a constant $M$. We let $\bar u:=u_{E^c \cup A}$ and we want to show that
\begin{equation}\label{O1}
\lim_{\eps \to 0}\frac{1}{|A|}\int_{B_1}(|\nabla \bar u|^2- |\nabla u|^2) \, dx =1.\end{equation}
 After modifying $f$ in the set $B_{2\eps} \setminus B_\eps$ we may assume that $f=g$ outside $B_{2\eps}$ and $f$ has bounded $C^{1,\gamma}$ norm. {F}rom \eqref{case1} we also obtain
that \begin{equation}\label{smn}
 \|  g\|_{C^{1,\frac \gamma 2}(B'_{2 \eps}) } ,\quad  \|  f\|_{C^{1,\frac \gamma 2}(B'_{2 \eps}) } \quad \mbox{are bounded by $ C \eps^\frac \gamma 2$.}
 \end{equation}
We have
$$\int_{B_1}|\nabla \bar u|^2- |\nabla u|^2 \, dx=\int_{B_1} \nabla (\bar u- u) \cdot \nabla (\bar u + u) \, dx.$$
After integrating by parts in the sets $E \setminus A$ and $A$ we find
\begin{equation}\label{dif}
\int_{B_1}|\nabla \bar u|^2- |\nabla u|^2 \, dx=\int_{\p A}  u \, \bar u_\nu \,  d\mathcal H^{n-1},
\end{equation}
with $\nu$ the exterior normal to $A$. We need to estimate
$$\int_{\Gamma} u \, \bar u_\nu \, d \mathcal H^{n-1} \quad \quad \mbox{with} \quad \Gamma:=\{(x',f(x')) {\mbox{ s.t. }} f(x') > g(x')\}.$$
Let $T \subset \Gamma$ be a measurable set and denote by $T'\subset \R^{n-1}$ its projection along $e_n$ direction.
Since in $B_\eps$, $u_n=1+o(1) $ with $o(1) \to 0$ as $\eps \to 0$, we use \eqref{smn} and we see that
$$ (1+o(1)) \inf_{T} \bar u_\nu \int_{T'}h \, dx' \le \int_T  u\, \bar u_\nu \,  d \mathcal H^{n-1} \le (1+o(1)) \sup_{T} \bar u_\nu \int_{T'}h \, dx',$$
with $$h:=f-g.$$

For the upper bound we use that $\bar u \le v$ with $v$ defined in Lemma \ref{c1ga}.
Then $\bar u_\nu \le v_\nu=1+o(1)$ in $\Gamma$ and we find that
\begin{equation}\label{L1}
\int_{\Gamma} u \, \bar u_{\nu} \, d \mathcal H^{n-1} \le (1+o(1))|A|. \end{equation}
For the lower bound we use Lemma \ref{env} for $h^+$ and consider its $C^{1,\gamma/2}$ envelope of norm $\eps^{\gamma /4} \gg \eps^{\gamma /2}$.
Denote by $K' \subset \R^{n-1}$ the contact set between $h^+$ and its envelope and let $K \subset \Gamma$ be the corresponding set that projects onto $K'$.

At any point $z\in K$ there is a $C^{1,\gamma/2}$ graph
$$G_z:=\{x_n=f_z(x')\} \quad \quad f_z:=g + l_z + \eps^{\frac \gamma 4}|x'-z'|^{1+\frac \gamma 2} ,$$
and $G_z$ is tangent by above to $A$ and is included in $E \setminus A$. Moreover after using a cutoff function we may assume $h_z$ has small $C^{1,\gamma/2}$ norm in a neighborhood of $0$ and coincides with $g$ outside this neighborhood. Let $v_z$ denote the corresponding harmonic function for $h_z$ as in Lemma \ref{c1ga}. Then $\bar u \ge v_z $, or $\bar u_\nu(z) \ge 1+o(1)$ and we obtain
\begin{equation}\label{L2}
\int_{K}   \bar u_\nu \, u \, \,  d \mathcal H^{n-1} \ge (1+o(1))\int_{K'} h \, dx' \ge (1+o(1)) \int_{\Gamma'} h \, dx' , \end{equation}
where in the last inequality we used Lemma \ref{env}.
Then \eqref{O1} follows from \eqref{L1} and \eqref{L2}.

\

\noindent{\it Case 2:} The set $A$ is exterior to $E$,
$$A=\{f(x') < x_n \le g(x') \} \subset B_\eps,$$
for some function $f$ with $C^{1,\gamma}$ norm bounded. We let $\bar u:=u_{E^c \setminus A}$ and we want to show that
\begin{equation}\label{O2}
\lim_{\eps \to 0}\frac{1}{|A|}\int_{B_1}(|\nabla u|^2- |\nabla \bar u|^2)  dx =1.\end{equation}
As before we may assume that $h=g$ outside $B_{2\eps}$ and \eqref{smn} holds. Since
\begin{equation}\label{A1}
\int_{B_1}|\nabla \bar u|^2- |\nabla u|^2 \, dx=\int_{\p A}  \bar u \, u_\nu \, d\mathcal H^{n-1}\end{equation}
and
\begin{equation}\label{A2}
u_\nu=1+o(1)\end{equation} we need to estimate
$$\int_{\Gamma} \bar u  \, d \mathcal H^{n-1} \quad \quad \mbox{with} \quad \Gamma:=\{(x',g(x')) {\mbox{ s.t. }} g(x') > f(x')\}.$$
The function $v$ defined in Lemma \ref{c1ga} is a lower barrier for $\bar u$ and since $v_n=1+o(1)$ we obtain
\begin{equation}\label{B1}
\int_{\Gamma} \bar u \, d \mathcal H^{n-1} \ge (1+o(1)) \int_{\Gamma'} h \, dx', \quad \quad {\mbox{with }}h:=(g-f)^+.\end{equation}
For the upper bound we apply Lemma \ref{env} for the function $h$ as in case 1 above.
For any $z=(z',f(z'))$, $z' \in \Gamma' $ we define the graph $G_z$ of the function
$$G_z:=\{x_n=f_z(x')\}, \quad \quad f_z:=g - l_z - \eps^{\frac \gamma 4}|x'-z'|^{1+\frac \gamma 2},$$
which is included in $E^c$ and it is tangent to $A$ by below at $z$. Since $\bar u \le v_z$ and $\p _{n} v_z=1+o(1)$ we obtain $$\bar u \le (1+o(1)) (x_n-f_z(x')).$$
After taking the infimum over all $z \in \Gamma$ we find $$\bar u(x',g(x_n)) \le (1+o(1))h^*(x') \quad \quad \forall x' \in \Gamma'.$$ By Lemma \ref{env} we find
\begin{equation}\label{B2}
\int_{\Gamma} \bar u \, \, d \mathcal H^{n-1} \le (1+o(1)) \int_{\Gamma'}h^* \, dx' \le (1+o(1)) \int_{\Gamma'}h \,  dx' .\end{equation}
Now, \eqref{O2} is a consequence of \eqref{A1}, \eqref{A2}, \eqref{B1} and \eqref{B2}, and this ends the proof of Lemma \ref{c1g}.~\qed

\end{document}